\documentclass[12pt, a4paper, leqno]{amsart}

\usepackage[utf8]{inputenc}
\usepackage{amsmath} 
\usepackage{amsfonts}
\usepackage{amssymb}
\usepackage{txfonts}   
\usepackage{amsfonts}     
\usepackage{graphicx}    
\usepackage[dvipsnames]{xcolor}
\usepackage[colorlinks, allcolors=RedViolet]{hyperref}
 
\newcommand{\R}{\varmathbb{R}}
\newcommand{\Z}{\varmathbb{Z}}
\newcommand{\N}{\varmathbb{N}}
\newcommand{\Rn}{{\varmathbb{R}^n}}

\newcommand{\Ha}{\mathcal{H}}
\newcommand{\M}{\mathcal{M}}

\newcommand{\ve}{\varepsilon}

\newcommand{\vint}{\operatornamewithlimits{\boldsymbol{--} \!\!\!\!\!\! \int }}

\def\ess{\qopname\relax o{ess}}
\def\diam{\qopname\relax o{diam}}
\def\loc{\qopname\relax o{loc}}

\def\dist{\qopname\relax o{dist}}
\def\min{\qopname\relax o{min}}
\def\diam{\qopname\relax o{diam}}

\def\phi{\varphi}
\def\s{{s(\cdot)}}
\def\p{{p(\cdot)}}
\newcommand{\It}{{\tilde I}}
\newcommand{\PP}{{\mathcal P}}
\newcommand{\PPln}{{\mathcal P^{\log}}}

\let\oldmarginpar\marginpar
\renewcommand\marginpar[1]{\-\oldmarginpar[\raggedleft\footnotesize #1]%
{\raggedright\footnotesize #1}}

\usepackage{amsthm}

\swapnumbers
\theoremstyle{plain}
\newtheorem{theorem}[equation]{Theorem}
\newtheorem{lemma}[equation]{Lemma}

\newtheorem{corollary}[equation]{Corollary}

\theoremstyle{definition}
\newtheorem{definition}[equation]{Definition}

\newtheorem{example}[equation]{Example}

\theoremstyle{remark}
\newtheorem{remark}[equation]{Remark}

\numberwithin{equation}{section}

\title[Estimates for the variable order Riesz potential]{Estimates for the variable order Riesz potential with applications}

\author{Petteri Harjulehto}
\address[Petteri Harjulehto]{Department of Mathematics and Statistics,
FI-20014 University of Turku, Finland}
\email{petteri.harjulehto@utu.fi}

\author{Ritva Hurri-Syrj\"anen}
\address[Ritva Hurri-Syrj\"anen]{Department of Mathematics and Statistics,
FI-00014 University of Helsinki, Finland}
\email{ritva.hurri-syrjanen@helsinki.fi}

\date{\today}

\begin{document}

\keywords{exponential integrability, Hausdorff content, non-smooth domain,
point-wise estimate,  Poincar\'e inequality,  Riesz potential, variable exponent}
\subjclass[2020]{46E35 (31C15, 42B20, 46E30, 26D10)}

\begin{abstract} 
We study weak-type estimates and exponential integrability for the variable order Riesz potential.
As an   application we prove an  exponential integrability result with respect to the Hausdorff content for functions  from variable exponent Sobolev spaces.
In particular,  the earlier exponential integrability results are improved 
to a corresponding one with respect to  the Choque integral whenever
 John domains are considered. Moreover, new exponential integrability results also for domains with outward cusps are obtained.
 \end{abstract}

\maketitle

\section{Introduction}

 If  $\Omega$  is  a bounded, open set in the Euclidean space $\R^n$,  $n\geq 2$, and
 $\alpha$  is a continuous function which satisfies $0< \alpha(x) <n$ for every $x \in \Omega$,
then we write  for the operator $I_{\alpha(\cdot)}$ acting on locally integrable functions $f$ in $\Omega$
\begin{equation}\label{RieszCond1}
 I_{\alpha(\cdot)}f(x) := \int_\Omega \frac{|f(y)|}{|x-y|^{n-\alpha(x)}} \, dy.
 \end{equation}
 This is called the variable order Riesz potential.
By defining the variable dimensional Hausdorff content
$\Ha_\infty^{\beta(\cdot)}$
 based on \cite[2.10.1, p.~169]{Federer} in Definition \ref{VarDimHaus} we
prove  that the level sets of the variable order Riesz potential $I_{\alpha(\cdot)}f$ 
are exponentially decaying with respect to the the variable dimensional Hausdorff content.
We show
\begin{theorem}\label{thm:Riesz-limit-case}
Let $\Omega$ be an open, bounded set in $\Rn$ and let
 $\alpha:\Omega \to [0, n)$ be a $\log$-Hölder continuous function with 
\begin{equation*}
0 < {\alpha}^-:= \ess\inf_{x\in \Omega}\alpha (x) \le
{\alpha}^+:= \ess\sup_{x\in \Omega} \alpha(x)<n.
\end{equation*}

Then, there exist constants $c_1$ and $c_2$  independent of the function $f$ such that the inequality
\[
\Ha_\infty^{n-\alpha(\cdot)}(\{x \in \Omega: I_{\alpha(\cdot)}f(x) >t\})
\le c_1\exp (-c_2t^{\frac{n}{n-\alpha^{-}}})
\]
holds
for all $f \in L^{n/\alpha(\cdot)}(\Omega)$ with  $\|f\|_{L^{n/\alpha(\cdot)}(\Omega)}\le \frac{1}{2(1 + |\Omega|)}$.
\end{theorem}
This result could be seen as a generalisation of
the work of \'{A}ngel D. Mart\'{i}nez and Daniel  Spector \cite[Theorem 1.2]{MarS21} to the variable order case.

As  an application
of Theorem \ref{thm:Riesz-limit-case} we prove
exponential integrability estimates for variable exponent Sobolev functions defined on bounded domains.
This result is our second main theorem, Theorem \ref{thm:main-p=n} where we generalise the earlier  known results of the exponential integrability
which were  with respect to the Lebesgue measure 
by taking  the integration with respect to the Hausdorff content 
as the  Choque integral as 
Mart\'{i}nez and  Spector  stated and proved for the usual Riesz potential in \cite{MarS21}.

We state an important corollary of  Theorem \ref{thm:main-p=n} which gives
 new  results to the $s$-John domains.
 Examples of these domains are convex domains and domains with Lipschitz boundaries but also 
domains  which are allowed to have outward cusps of the order $s$.
  
\begin{corollary}  \label{cor:main}
 Let $D$ be an  $s$-John domain in $\R^n$, $n\geq 2,$ with
 $1\le s < \frac{n}{n-1}$. Then, there exist positive constants $a$ and $b$ such that
 \[
\int_D \exp\big( a |u-u_B|^{\frac{n}{s(n-1)}} \big) \, d \Ha^{s(n-1)}_\infty \le b
\]
 for all $u \in L^{1}_{n/(n-s(n-1))}(D)$ with $\|\nabla u\|_{L^{n/(n-s(n-1))}(D)} \le 1$, 
 $B:= B(x_0, \dist(x_0, \partial D))$.
\end{corollary}

 In particular  this corollary   generalises
  \cite[Theorem 2]{Tru1967} and \cite[Theorem 3.4]{EH-S2001} to the Choque integral case when $s=1$,
   that is, for John domains.
  Corollary  \ref{cor:main} gives a complete new result whenever $1<s<n/(n-1)$.

We also prove a weak-type estimate for the variable order Riesz potential in Theorem~\ref{hteJ}.
By this result Poincar\'e-type inequalities with respect to  the Luxemburg norm on the variable exponent Lebesgue space  are proved to be valid for $L^1_p$-functions defined
on a  new class of domains in
Theorems \ref{thm:main-p>1} and
\ref{thm:main-p=1}. The new class of domains suits well to the questions for variable exponent Lebesgue and Sobolev Spaces.

This paper is organized as follows.
Definitions for variable exponent Lebesgue spaces and Sobolev spaces are recalled in Section 2. Estimates for the variable order Riesz potential acting on
functions from the variable order Lebesgue spaces are given and proved in Section 3.
The variable dimensional Hausdorff content is defined and its basic properties are proved in Section 4.
The variable order maximal function is recalled there and point-wise  estimates for the Riesz potential are proved there, too.
Our main weak-type estimate Theorem \ref{thm:Riesz-limit-case} is proved in Section 4.
Applications are considered in Section 5.
We start by defining a new class of functions which includes $s$-John domains and then state and prove
Poincar\'e-type inequalities of 
Theorems \ref{thm:main-p>1} and
\ref{thm:main-p=1}.
there.
Our main theorem in Section 5  is on  the exponential integrability, 
Theorem \ref{thm:main-p=n} which yields Corollary \ref{cor:main}.

\section{Notations}

Let $U$ in $ \Rn$ be any bounded, open set. For  any   measurable function $g:U \to \R$ and  measurable set $A\subset U$ we define
\[
g^+_A := \ess\sup_{x\in A} g(x) 
\quad\text{and}\quad 
g^-_A := \ess\inf_{x\in A} g(x)\,.
\]
In the case $A=U$ we write $g^+ := g^+_U$ and   $g^- := g^-_U$.
We say that  the function $g:U \to \R$ satisfies
the \textit{ $\log$-H\"{o}lder continuity} condition if
 there is a constant $c_1>0$ such that
\begin{equation}\label{HolderCond1}
|g(x)-g(y)|\leq \frac{c_1}{\log(e+1/|x-y|)}
\end{equation}
for all $x,y\in U$, $x \neq y$.  
Note that $g$ is  a  $\log$-H\"{o}lder continuous function if and only if $|B|^{g^-_{B\cap U} - g^+_{B\cap U}}\lesssim 1$ for all open balls $B \cap U \neq \emptyset$.

By a \textit{variable exponent} we mean a measurable function $p:U \to [1,
\infty)$ such that $1\le p^-\le p^+ <\infty$.
The set of all variable exponents is denoted by $\PP(U)$.
By $\PPln(U)$ we denote a subset consisting of  all 
$\log$-H\"{o}lder continuous variable exponents.

We define a \emph{modular} on the set of Lebesgue measurable functions  $f$ by
setting
\[
\varrho_{L^{p(\cdot)}(U)}(f) :=\int_{U} |f(x)|^{p(x)}\,dx.
\]
The \emph{variable exponent Lebesgue space $L^{p(\cdot)}(U)$}
consists of all measurable functions $f\colon U\to \R$ for which the
modular $\varrho_{L^{p(\cdot)}(U)}(f)$ is finite. 
We write $f\in L^{p(\cdot )}(U)$.

The Luxemburg--Nakano norm on this space is defined as
\[
\|f\|_{L^{p(\cdot)}(U)}:=
\inf\Big\{\lambda > 0\,\colon\, \varrho_{L^{p(\cdot)}(U)}\big(f/\lambda\big)\leq 1\Big\}.
\]
Equipped with this norm, $L^{p(\cdot)}(U)$ is a Banach space.
We use the abbreviation $\|f\|_\p$ to denote the norm in the whole space.

For open sets $U$, the  variable exponent \emph{Sobolev space
  $L^{1}_{p(\cdot)}(U)$} consists of functions $u\in L^{1}_{\loc}(U)$
for which the  absolute value of the distributional gradient $|\nabla u|$ belongs to
$L^{p(\cdot)}(U)$:
\[
L^{1}_{\p}(U) := \{u \in W^{1, 1}_{\text{loc}}(U): |\nabla u| \in L^\p(U)\},
\]
where $W^{1, 1}_{\text{loc}}(U)$ is the classical local Sobolev space.

More information and proofs for these facts can be found 
in
\cite[Chapters~2,~4,~8, and~9]{DieHHR11} or from \cite{CruZ13}.

\section{Strong- and weak-type estimates}

 Let $\Omega$  be  a bounded, open set in the Euclidean space $\R^n$,  $n\geq 2$.
  We assume that a continuous function $\alpha$ satisfies $0< \alpha(x) <n$ for every $x \in \Omega$.
  We recall  the definition 
  \eqref{RieszCond1}
  from Introduction.
 The potential  in \eqref{RieszCond1}  is called the variable order Riesz potential
  of a function $f$.

If $\alpha$ is  a $\log$-Hölder continuous function, then  we have
\[
 I_{\alpha(\cdot)}f(x) \approx \int_\Omega \frac{|f(y)|}{|x-y|^{n-\alpha(y)}} \, dy,
\]
we refer to  \cite[p. 270]{Has09}.
Let us define
\[
\frac1{p^\#_\alpha(x)} := \frac1{p(x)}- \frac{\alpha (x)}{n} \quad \text{i.e.} \quad  p^\#_\alpha(x) = \frac{np(x)}{n-\alpha(x) p(x)}.
 \]

The variable order Riesz potential has been studied  extensively, see for example \cite{EdmM02, KokM07, MizNOS12, Sam98a, Sam98} and references  therein.
S. Samko has proved the following result.

\begin{lemma}[Theorem~3.2 of \cite{Sam98}]\label{lem:Samko}
  Let $\Omega \subset \Rn$ be  a bounded, open set.
   Let $p \in \PPln(\Omega)$ satisfy $1<p^-\le p^+<\infty$.
  Assume that $\alpha^- >0$  and   $(\alpha p)^+<n$.
 Then $I_{\alpha(\cdot)} : L^{\p}(\Omega) \to L^{p^\#_\alpha(\cdot)}(\Omega)$ is bounded, where $p^\#_\alpha(x) = \frac{np(x)}{n-\alpha(x) p(x)}$. 
 \end{lemma}

 Let us consider  the case $p^-=1$. It is well known that $I_1$, where $\alpha\equiv 1$, does not map $L^1(\Omega) \to L^1(\Omega)$. 
 Thus instead of the strong-type estimate we can have only a weak-type estimate. 
For that we need to assume that $\alpha$ is $\log$-Hölder continuous. The following theorem is a modification of
\cite[Theorem 4.3]{AlmHHL15}. We consider a bounded set and obtain a better control for the extra term  than in \cite[Theorem 4.3]{AlmHHL15}.

\begin{theorem}\label{hteJ}      
  Let $\Omega\subset\Rn$ be a bounded, open set.  Let $p \in \PPln(\Omega)$ satisfy $1 \le p^-\le p^+<\infty$. Assume that  $\alpha$ is  a $\log$-Hölder continuous  function with  $\alpha^- >0$  and   $(\alpha p)^+ <n$.
 Let $p^\#_\alpha(x) = \frac{np(x)}{n-\alpha(x) p(x)}$. Then  there exists a constant $c$ such that for every $f\in L^{p(\cdot)}(\Omega)$ with
  $\Vert f\Vert_{p(\cdot)} \le 1$ and for every $t>0$ 
  the inequality
  \[ 
  \int_{\{x \in \Omega : I_{\alpha(\cdot)} f >t\}} t^{p^\#_\alpha (x)} dx \leq c
  \int_{\Omega} |f(y)|^{p(y)} \, dy + c\big|\{x\in\Omega: 0<|f(x)| \leq
  1\}\big|
  \]
  holds. The constant $c$ depends only  $n$, $\diam(\Omega)$, $\alpha^-$, $(\alpha p)^+$, $p^+$ and $\log$-Hölder constants of $p$ and $\alpha$.
\end{theorem}

 \begin{proof}
As a part of the proof of \cite[Theorem~3.2, p.~278 ]{Sam98}  S. Samko proved the following point-wise inequality:
\begin{equation}\label{Samkos_pointwise}
|I_{\alpha(\cdot)} f(x)|\le c |\M f(x)|^{\frac{p(x)}{p^\#_\alpha(x)}} 
\end{equation}
 for almost all $x \in \Omega$.
Here $\M$ is the  Hardy--Littlewood maximal operator
\[
\M f(x) := \sup_{r>0} \frac{1}{|B(x,r)|}\int_{B(x,r) \cap \Omega} |f| \, dy\,,
\]
and the constant  $c$ depends only on the dimension $n$, $\alpha^-$, $(\alpha p)^+$ and $p^+$.
The proof is based on Hedberg's trick,  \cite{Hed72}.

Thus we have  
\[
  \{x \in \Omega: I_{\alpha(\cdot)} f(x) >t\} 
  \subset \Big\{x \in \Omega: c \,[\M f(x)]^{\frac{p(x)}{p^\#_\alpha(x)}} >t\Big\}=:E.
  \]
  Let us write 
  \[
  A_{B} f :=  \frac{1}{|B]} \int_{B\cap \Omega} |f| \, dy. 
  \]
  
  For every $z\in E$ we choose $B_z:= B(z,r_z)$ such that $\displaystyle
  c \, (A_{B_z} f)^{\frac{p(z)}{p^\#_\alpha(z)}} > t$ 
  where  $c$ is the constant from \eqref{Samkos_pointwise}.
   Let $x\in
  B_z$ and let us raise this inequality to the power $p^\#_\alpha(x)$.
  Assume first that  $\frac{p^\#_\alpha(x)p(z)}{p^\#_\alpha(z)} \ge p(x)$.  
   From now on  $c$  may vary from line to line. Since
 $\|f\|_\p \le 1$ we get $A_{B_z}
  f   \le  c |B_z|^{-1}$, and thus we obtain
  \begin{align*}
    t^{p^\#_\alpha(x)} &\le c \, (A_{B_z} f)^{p(x)}
    (A_{B_z} f)^{\frac{p^\#_\alpha(x)p(z)}{p^\#_\alpha(z)}-p(x)} \le c \, (A_{B_z} f)^{p(x)}
    |B_z|^{p(x)-\frac{p^\#_\alpha(x)p(z)}{p^\#_\alpha(z)}}\\
    &= c\,(A_{B_z} f)^{p(x)} |B_z|^{\frac{p(x)p^\#_\alpha(z)-p^\#_\alpha(x)p(z)}{p^\#_\alpha(z)}} = c\,(A_{B_z} f)^{p(x)} |B_z|^{p(x)- p(z)} |B_z|^{\frac{p(z)}{p^\#_\alpha(z)}(p^\#_\alpha(z) -p^\#_\alpha(x))}.
  \end{align*}
 If $p(x) - p(z) >0$, then $|B_z|^{p(x)- p(z)}$ is uniformly bounded by $c(n)\diam(\Omega)^{np^+}$.
  Otherwise we use  $\log$-H{\"o}lder continuity of $p$ and obtain  that $|B_z|^{p(x)- p(z)}$ is uniformly bounded by \cite[Lemma 4.1.6, p.~101]{DieHHR11}.
  Let us look  at the last term  in the previous estimate. A short calculation gives that
  \[
  p^\#_\alpha(z) -p^\#_\alpha(x) = \frac{n^2(p(z)-p(x))+ np(x)p(z)(\alpha(z)-\alpha(x))}{(n-\alpha(z) p(z))(n-\alpha(x) p(x))}.
  \]
   Since $p$ and $\alpha$  are  bounded $\log$-Hölder continuous functions  and 
    \[
  0<c_1 \le (n-\alpha(z) p(z))(n-\alpha(x) p(x))\le n^2<\infty ,
  \]
   with some positive constant $c_1$, 
  the term $|B_z|^{\frac{p(z)}{p^\#_\alpha(z)}(p^\#_\alpha(z) -p^\#_\alpha(x))}$ is
  uniformly bounded.  
   By \cite[Theorem 4.2.4, p.~108]{DieHHR11} we can move the exponent $p(x)$ inside  the integral and obtain
   \[
   t^{p^\#_\alpha(x)} \le c\,A_{B_z}( |f|^{\p} + \chi_{\{0<|f|\le 1\}}).
   \]
   
   Let us assume now that $\frac{p^\#_\alpha(x)p(z)}{p^\#_\alpha(z)}< p(x)$.  Again by \cite[Theorem 4.2.4, p.~108]{DieHHR11} we obtain
  \[
  \begin{split}
  t^{p^\#_\alpha(x)}
  &\le c \, \Big( A_{B_z} \big(|f| \big) \Big)^{p(x) \frac{p^\#_\alpha(x)p(z)}{p^\#_\alpha(z)p(x)}} \le c\,
  \bigg(A_{B_z}\big(|f|^{\p} + \chi_{\{0<|f|\le 1\}} \big)\bigg)^{\frac{p^\#_\alpha(x)p(z)}{p^\#_\alpha(z)p(x)}}\\
 &\le c\,
  A_{B_z}\big(|f|^{\p} + \chi_{\{0<|f|\le 1\}} \big),
  \end{split}
  \]
  where the last inequality follows from the estimates 
  $A_{B_z}\big(|f|^{\p} + \chi_{\{0<|f|\le 1\}} \big)\ge 1$ 
  and $(p^\#_\alpha(x)p(z))/(p^\#_\alpha(z)p(x)) < 1$.

  By the Besicovitch covering theorem, \cite[2.7]{Mattila}, there
  is a countable covering subfamily $(B_i)$, with the bounded
  overlapping-property.  Thus we have
  \begin{align*}
    \int_E t^{p^\#_\alpha(x)} dx &\le \sum_{i=1}^\infty \int_{B_i \cap \Omega}
    t^{p^\#_\alpha(x)} dx\\ 
    &\le c \sum_{i=1}^\infty \int_{B_i} \, \frac{1}{|B_i|}\int_{B_i \cap \Omega}
    |f(y)|^{p(y)} +\chi_{\{0<|f|\le 1\}}(y)~ dy~ dx
    \\
    &= c \sum_{i=1}^\infty \int_{B_i\cap \Omega} |f(y)|^{p(y)} +\chi_{\{0<|f|\le 1\}}(y)~ dy
    \\
    &\le c \int_{\Omega} |f(y)|^{p(y)} dy + c|\{x\in\Omega:0<|f|\le 1\}|. 
    \qedhere
  \end{align*}
\end{proof}

\begin{remark}
If  $\alpha$ is a constant and $p\equiv 1$ in Theorem ~\ref{hteJ}, then we obtain a weak type estimate.
Indeed,  in this case $p^\#_\alpha=\frac{n}{n-\alpha}$ is a constant  and we have
\[
|\{x \in \Omega : I_{\alpha(\cdot)} f >t\}| t^{p^\#_\alpha} = \int_{\{x \in \Omega : I_{\alpha(\cdot)} f >t\}} t^{p^\#_\alpha} dx \leq c (1+ |\Omega|)
\]
for some constant $c$.
Hence the quasinorm
\[
\|f\|_{L_w^{p^\#_\alpha}(\Omega)}:= \sup_{t>0} |\{I_{\alpha(\cdot)} f >t\}|^{1/p^\#_\alpha} t
\] 
is uniformly bounded for any function satisfying $\|f\|_1 \le 1$.  Using a function $f/\|f\|_1$ we obtain
$\frac{\|f\|_{L_w^{p^\#_\alpha}(\Omega)}}{\|f\|_1}=\big\|\frac{f}{\|f\|_1}\big\|_{L_w^{p^\#_\alpha}(\Omega)} \le c$.
Thus, the weak-type estimate 
\[
\sup_{t>0} |\{I_{\alpha(\cdot)} f >t\}|^{1/p^\#_\alpha} t \le c \|f\|_1
\]
holds.
\end{remark}

 \begin{remark}\label{RieszNewNotation}
 
 In Section~\ref{John} we need the operator
 
 \begin{equation*}
 \It_{s(\cdot)}f(x) := \int_\Omega \frac{|f(y)|}{|x-y|^{s(x)(n-1)}} \, dy,
 \end{equation*}
 where $1\le s(x) <\frac{n}{n-1}$ for every $x \in \Omega$.
 We have $I_{\alpha(\cdot)} = \It_\s$, where
  $\alpha(x) := n-s(x)(n-1)$. 
  We see that $\alpha$ is $\log$-Hölder continuous if and only if  $s$ is $\log$-Hölder continuous.
  We recover the assumptions for $s$:
  $\alpha^- >0$ if and only if $s^+ < \frac{n}{n-1}$.
  $(\alpha p)^+ <n$ if and only if $\sup_{x \in \Omega} (n-s(x)(n-1))p(x)<n$.
 Moreover
 \begin{align*}
q(x) &:= p^\#_{n-s(n-1)} = \frac{np(x)}{n-np(x)+s(x)p(x)(n-1)}\\
&=\frac{np(x)}{p(x)n(s(x)-1)+n -s(x) p(x)}.
 \end{align*}
\end{remark}

\section{Exponential integrability}

In this section we study behaviour of   the variable order Riesz potential  $I_{\alpha(\cdot)} f$ in the limiting case $f \in L^{\frac{n}{\alpha(\cdot)}}$. We follow a recent article
by Mart{\'i}nez and  Spector \cite{MarS21} where they studied the usual Riesz potential, that is, the constant order Riesz potential  in $\Rn$.

We define the variable dimensional Hausdorff content based on \cite[2.10.1, p.~169]{Federer}.
 For  the definition of the usual Hausdorff content we refer to \cite{Ad1998}, too.

\begin{definition}[Variable dimensional Hausdorff content]\label{VarDimHaus}
Let $\beta:\Rn \to (0, n]$ be a function  and $E$ any  set in $\Rn$. Then  we write
\[
\Ha_\infty^{\beta(\cdot)} (E) := \inf \bigg\{ \sum_{i=1}^\infty r_i^{\beta (x_i)}: E \subset \bigcup_{i=1}^\infty B(x_i, r_i)\bigg\}.
\]
\end{definition}

Here $\Ha_\infty^{\beta (\cdot)} (E)$ is called the variable dimensional Hausdorff content of the set $E$, and in the next lemma we show that it is an outer capacity.
 Hence, based on \cite{AH} we could call 
$\Ha_\infty^{\beta(\cdot)}$  also a variable Hausdorff capacity.  
We refer to \cite{AHS}, too.
In the next lemma  $\varmathbb{P}(\Rn)$ is the power set.

\begin{lemma}\label{lem:H}
Let $\beta:\Rn \to (0, n]$.
The set function $\Ha_\infty^{\beta(\cdot)}: \varmathbb{P}(\Rn) \to [0, \infty]$ satisfies the following properties:

\begin{enumerate}
\item $\Ha_\infty^{\beta(\cdot)}(\emptyset) =0$;
\item if $A \subset B$ then $\Ha_\infty^{\beta(\cdot)}(A) \le \Ha_\infty^{\beta(\cdot)} (B)$;
\item  if $E \subset \Rn$ then 
\[
\Ha_\infty^{\beta(\cdot)}(E) = \inf_{E \subset U \text{ and }U \text{ is open}}\Ha_\infty^{\beta(\cdot)}(U); 
\]
\item if $(K_i)$ is a decreasing sequence of compact sets then 
\[
\Ha_\infty^{\beta(\cdot)}\Big(\bigcap_{i=1}^\infty K_i \Big)
= \lim_{i \to \infty} \Ha_\infty^{\beta(\cdot)}(K_i);
\]
\end{enumerate}
\end{lemma}

\begin{proof}
The claim (1) is clear. The claim (2) follows since every cover of $B$ is also a cover of $A$.

Let us next prove (3). By (2) we have $\Ha_\infty^{\beta(\cdot)}(E) \le \inf_{U}\Ha_\infty^{\beta(\cdot)}(U)$.
Let $\ve>0$, and choose a covering that satisfies $E\subset \bigcup_{i=1}^\infty B(x_i, r_i)$ and
\[
 \sum_{i=1}^\infty r_i^{\beta(x_i)} < \Ha_\infty^{\beta(\cdot)}(E) + \ve. 
 \]
We choose that  $V := \bigcup_{i=1}^\infty B(x_i, r_i)$ and note that it is open. Thus 
\[
\inf_{U}\Ha_\infty^{\beta(\cdot)}(U) \le \Ha_\infty^{\beta(\cdot)}(V) \le  \sum_{i=1}^\infty r_i^{\alpha(x_i)} < \Ha_\infty^{\beta(\cdot)}(E) + \ve. 
\]
 Since this holds for all $\ve>0$, we obtain $\inf_{U}\Ha_\infty^{\beta(\cdot)}(U) \le \Ha_\infty^{\beta(\cdot)}(E)$.
 Hence the claim (3) is proved.
   
Let us then prove (4). By (2) we have $\Ha_\infty^{\beta(\cdot)}\big(\bigcap_{i=1}^\infty K_i \big)
\le \lim_{i \to \infty} \Ha_\infty^{\beta(\cdot)}(K_i)$, and the limit exists by (2) since the sequence  of
$(K_i)$ is decreasing.
Let $\ve>0$, and choose by (3) an open set $V$ that satisfies $\bigcap_{i=1}^\infty K_i \subset V$ and 
\[
 \Ha_\infty^{\beta(\cdot)}(V) < \Ha_\infty^{\beta(\cdot)}\Big(\bigcap_{i=1}^\infty K_i \Big) + \ve. 
 \]
 Since $V$ is  open we find $j_0 \in \N$ such that  $K_j \subset V$
  for all $j \ge j_0$. Thus for all $j \ge j_0$ we have
  \[
\Ha_\infty^{\beta(\cdot)}(K_j) \le  \Ha_\infty^{\beta(\cdot)}(V) < \Ha_\infty^{\beta(\cdot)}\Big(\bigcap_{i=1}^\infty K_i \Big) + \ve,
 \]
 and furthermore $ \lim_{i \to \infty} \Ha_\infty^{\beta(\cdot)}(K_i) \le \Ha_\infty^{\beta(\cdot)}\big(\bigcap_{i=1}^\infty K_i \big) + \ve$. 
 Since this holds for all $\ve>0$, we obtain the inequality for the other direction. Hence also the claim (4) is proved.
 \end{proof}

\begin{remark}
We do not know is $\Ha_\infty^{\beta(\cdot)}$ a capacity in the sense of Choquet \cite{Cho53}. More precisely we dot not know does the following hold:
 if $(E_i)$ is a increasing sequence of  sets then 
$\Ha_\infty^{\beta(\cdot)}\big(\bigcup_{i=1}^\infty E_i \big)
= \lim_{i \to \infty} \Ha_\infty^{\beta(\cdot)}(E_i)$.
If $\beta$ is a constant function, then this property has been proved in \cite{Dav56}, see also \cite{Dav70, SioS62}. 
\end{remark}

We recall the definition of the variable order fractional maximal function, \cite{KokM07}.

\begin{definition}[Variable order fractional maximal function]
Let $\alpha:\Omega \to [0, n)$  any measurable function  and $f \in L^1_{\loc}(\Omega)$. Then
\[
\M_{\alpha(\cdot)}f(x)  := \sup_{r>0} \frac{r^{\alpha(x)}}{|B(x, r)|} \int_{B(x,r) \cap \Omega} |f(y)| \, dy\,, \quad x\in\Omega .
\]
\end{definition}

We give a proof  for the lower semicontinuity of this maximal function for the reader's convenience.

\begin{lemma}
Let $\Omega\subset \Rn$ be an open set, and $\alpha: \Omega \to (0, n]$ be continuous. Then
$\M_{\alpha(\cdot)}f$ is lower-semicontinuous.
\end{lemma}

\begin{proof}
Let us write $E_t:= \{x \in \Omega: \M_{\alpha(\cdot)}f(x) >t\}$. We need to show that $E_t$ is open in $\Rn$.
So let $x \in E_t$. By the definition of $\M_{\alpha(\cdot)}$ for every $x \in E_t$ there exists a radius $r_x>0$ such that
\[
\frac{r_x^{\alpha(x)}}{|B(x, r_x)|} \int_{B(x,r_x) \cap \Omega} |f(y)| \, dy >t.
\]
By the properties of  the Lebesgue measure we obtain
\[
\lim_{R\searrow r_x}\frac{R^{\alpha(x)}}{|B(x, R)|} \int_{B(x,r_x) \cap \Omega} |f(y)| \, dy = \frac{r_x^{\alpha(x)}}{|B(x, r_x)|} \int_{B(x,r_x) \cap \Omega} |f(y)| \, dy.
\]
Hence there exists $R>r_x$ such that 
\[
\lambda:= \frac{\frac{R^{\alpha(x)}}{|B(x, R)|} \int_{B(x,r_x) \cap \Omega} |f(y)| \, dy}{t} >1.
\]
Let $x'$ be such that $|x-x'|< R-r_x$ and assume that  
\begin{equation}\label{condition}
|\alpha(x') - \alpha(x)|< \frac{\log(\lambda)}{\max\{1,|\log(R)|\}}.
\end{equation}
Then
\[
R^{\alpha(x')-\alpha(x)}
\ge \begin{cases}
1, &\text{ if } R\ge 1 \text{ and } \alpha(x') \ge \alpha(x), \\
R^{- \frac{\log(\lambda)}{\max\{1,\log(R)\}}} \ge R^{- \frac{\log(\lambda)}{\log(R)}} = \frac1\lambda &\text{ if } R\ge 1 \text{ and } \alpha(x') < \alpha(x), \\
R^{\frac{\log(\lambda)}{\max\{1, -\log(R)\}}} \ge R^{- \frac{\log(\lambda)}{\log(R)}} = \frac1\lambda &\text{ if } R< 1 \text{ and } \alpha(x') \ge \alpha(x), \\
1, &\text{ if } R< 1 \text{ and } \alpha(x') < \alpha(x). \\
\end{cases}
\]
The condition \eqref{condition}
holds in some ball $B(x, \xi) \cap \Omega$ by the  continuity of $\alpha$. Thus for all $x' \in B(x, \min\{R-r_x, \xi\}) \cap \Omega$ we have 
\[
\begin{split}
\M_{\alpha(\cdot)}f(x') &\ge \frac{R^{\alpha(x')}}{|B(x', R)|} \int_{B(x',R) \cap \Omega} |f(y)| \, dy\\
& \ge R^{\alpha(x')-\alpha(x)} \frac{R^{\alpha(x)}}{|B(x', R)|} \int_{B(x,r_x) \cap \Omega} |f(y)| \, dy\\
&> \frac{\lambda t}{\lambda} =t. 
\end{split}
\]
Hence  the set $E_t$ is open in $\Omega$, and thus  it is open also in $\Rn$.
\end{proof}

The following lemma is a generalisation of \cite[Lemma 3.5]{MarS21} and \cite[Theorem (ii)]{OV} to the case of 
the variational dimensional Hausdorff content.

\begin{lemma}\label{lem:Haudorff-1}
Let $\Omega \subset\Rn$ be a  bounded, open set.
Let $\alpha:\Omega \to [0, n)$  be a measurable function such that  $\alpha^+<n$.
Then there exists a constant $c$, depending only on the dimension $n$,  such that the inequality
\[
\Ha_\infty^{n-\alpha(\cdot)}(\{x \in \Omega: \M_{\alpha(\cdot)}f(x) >t\})
\le  \frac{c(n)}{t} \|f\|_{L^1(\Omega)}
\]
holds for all $f \in L^1(\Omega)$.
\end{lemma}

\begin{proof}
Let us write $E_t:= \{x\in \Omega: \M_{\alpha(\cdot)}f(x) >t\}$. 

By the definition of $\M_{\alpha(\cdot)}$ for every $x \in E_t$ there exists a radius $r_x>0$ such that
\[
\frac{r_x^{\alpha(x)}}{|B(x, r_x)|} \int_{B(x,r_x) \cap \Omega} |f(y)| \, dy >t.
\]
This  inequality yields that $r_x^{n- \alpha(x)} < \frac{c(n)}{t}  \|f\|_{L^1(\Omega)}$, and hence $r_x$ is uniformly bounded 
provided that $\alpha^+ <n$. 

Then  $E_t \subset \bigcup_{x \in E_t} B(x, r_x)$,
and hence by  the $5r$-covering lemma \cite{Stein} or 
the simple Vitali covering lemma \cite[Theorem 1.4.1]{AH} we find a countable subcollection of disjoint balls $(B(x_i, r_{x_i}))$ such that

\[
E_t \subset \bigcup_{i=1}^\infty  B(x_i, 5r_{x_i}).
\]
Thus by the definition of the Hausdorff content, we obtain
\[
\begin{split}
\Ha_{\infty}^{n-\alpha(\cdot)} (E_t) &\le \sum_{i=1}^{\infty} (5 r_i)^{n-\alpha(x_i)}\\
&\le  \frac{5^n}{\omega_n t} \sum_{i=1}^{\infty}\int_{B(x_i,r_{x_i}) \cap \Omega} |f(y)| \, dy
\le  \frac{5^n}{\omega_n t} \|f\|_{L^1(\Omega)},
\end{split}
\] 
where $\omega_n$ is the Lebesgue measure of the unit ball.
\end{proof}

The proof of the next lemma follows the proof of \cite[Proposition 2.1]{HarH08}.

\begin{lemma}
  \label{lem:rieszvsM2}
 Let $\Omega\subset \Rn$ be a bounded, open set.  Let $p \in \PPln(\Omega)$ satisfy the inequalities $1 \le p^-\le p^+<\infty$.
Assume that  $\alpha:\Omega \to (0, n)$ satisfies
  $\alpha^- >0$  and   $( \alpha p )^+ <n$.
 Let  $f \in
  L^{p(\cdot)}(\Rn)$ with $\|f\|_{p(\cdot)}\le 1$. Then  there exists a constant $c$ such that for every $x\in \Omega$ and every $r>0$ the inequality
  \begin{align*}
    \int_{\Omega \setminus B(x,r)}
    \frac{|f(y)|}{|x-y|^{n-\alpha(y)}} \,dy \leq c \; \max\Big\{1, \frac{p(x)}{n-\alpha(x) p(x)}\Big\}^{\frac{p^+ -1}{p^+}} r^{-\frac{n-\alpha(x) p(x)}{p(x)}}
  \end{align*}
holds. Here the constant $c$ depends only on the dimension $n$, the $\log$-Hölder constant of $p$ and $\diam(\Omega)$.
\end{lemma}

\begin{proof}
Let us denote  by $A(x,r)$ the annulus $(B(x,r) \setminus B(x,r/2)) \cap \Omega$
and  write $I:= \{i \in \N : r\le 2^i \le  \diam(\Omega)\}$.

Let us first note that Lemma 4.16 and Theorem 4.5.7 of \cite{DieHHR11}
yield that  $\|1\|_{L^{p(\cdot)}}(B) \le c |B|^{1/p(x)}$ for all $x \in B\cap \Omega$ and all ball $B$ with $\diam(B) \le \diam(\Omega)$. Here the constant $c$ depends only on $n$, $\log$-Hölder constant of $p$ and $\diam(\Omega)$.

When we use in \eqref{holderSumEq} H\"older's inequality for the second inequality,
 for the norm of the constant one for the third inequality, and finally Hölder's inequality again, 
we conclude that
\begin{equation}\label{holderSumEq}
\begin{split}
	\int_{\Omega \setminus B(x,r)}
    \frac{|f(y)|}{|x-y|^{n-\alpha(x)}} \,dy &\le\sum_{i\in I} 2^{i(\alpha(x)-n)} \int_{A(x, 2^i)} |u(y)| \, dy\\
	&\lesssim  \sum_{i \in I} 2^{i(\alpha(x)-n)}
	\| u\|_{L^{p(\cdot)}(A(x,2^i))} \| 1\|_{L^{p'(\cdot)}(B(x, 2^i)
	    \cap \Omega)} \\
	  & \lesssim \sum_{i\in I} 2^{i(\alpha(x)-n+ \frac{n}{p'(x)})}
	  \| u\|_{L^{p(\cdot)}(A(x,2^i))} \\
	  & \lesssim
	  \Bigg(\sum_{i\in I} 2^{-i n \frac{(p^+)'}{p^\#_\alpha(x)}}\Bigg)^{\frac1{(p^+)'}}
	  \Bigg(\sum_{i\in I}  \| u\|_{L^{p(\cdot)}(A(x,2^i))}^{p^+}
	  \Bigg)^{\frac1{p^+}}
\end{split}
\end{equation}
for $x\in \Omega$ where $p^\#_\alpha(x)= \frac{np(x)}{n-\alpha(x)p(x)}$. 
Since $\|u\|_{p(\cdot)}\le 1$ we have 
$\|u\|_{p(\cdot)}^{p^+} \le \varrho_{p(\cdot)}(u)$ by \cite[Lemma 3.2.5. p.~75]{DieHHR11}, 
and so
\[
\begin{split}
	\sum_{i\in I}  \| u\|_{L^{p(\cdot)}(A(x,2^i))}^{p^+}
	&\le  \sum_{i\in I} \int_{A(x,2^i)} |u(y)|^{p(y)}\, dy
	\le \int_\Omega |u(y)|^{p(y)}\, dy \le  1.
\end{split}
\]
The first term on the  last line  of \eqref{holderSumEq} 
is a geometric sum. Thus we obtain that 
\[
\Bigg(\sum_{i\in I} 2^{-i n \frac{(p^+)'}{p^\#_\alpha(x)}}\Bigg)^{\frac{1}{(p^+)'}}
\le r^{-\frac{n}{p^\#_\alpha(x)}} \bigg(1-2^{-n\frac{(p^+)'}{p^\#_\alpha(x)}}
\bigg)^{-\frac{1}{(p^+)'}}.
\]	 
Let us write that $k(x):=\max\{1, \frac{p(x)}{n-\alpha(x) p(x)}\}$.
Now $n/p^\#_\alpha(x) \ge 1/k(x)$ and $k(x)\ge 1$. Thus by the inequality
$x^a\le ax+1-a$ (which follows from Bernoulli's inequality) with $x=2^{-(p^+)'}$ and $a=\frac1k$ we obtain 
\[
\big(1-2^{-n\frac{(p^+)'}{p^\#_\alpha(x)}}\big)^{-1}\le 
\big(1-2^{-\frac{(p^+)'}{k(x)}}\big)^{-1} \le  \big(1-2^{-(p^+)'}\big)^{-1} k(x). 
\]
Hence we have
\[
\Bigg(\sum_{i\in I} 2^{-i n \frac{(p^+)'}{p^\#_\alpha(x)}}\Bigg)^{\frac{1}{(p^+)'}}
\le  \big(1-2^{-(p^+)'}\big)^{ - \frac{1}{(p^+)'}} k(x)^{\frac{1}{(p^+)'}} \, r^{-\frac{n}{p^\#_\alpha(x)}}. 
\]
Finally, when  we note by $(p^+)' \in (1, \infty)$ that $ \big(1-2^{-(p^+)'}\big) \in (\frac12, 1)$. Hence we have
the inequality
$\big(1-2^{-(p^+)'}\big)^{ - \frac{1}{(p^+)'}} \le 2$.
\end{proof}

In the variable exponent case the Hedberg-type estimates are well known and variants have been used and proved for example in
\cite[Theorem 3.8]{Die04},
\cite[Proposition 6.1.6]{DieHHR11}, \cite[(4.7)]{HarHL06}, \cite[p. 429]{MizNOS12}, \cite[Lemma 4.6]{MizS05}, \cite[p.~279]{Sam98}.
Here instead of the standard maximal operator we have the variable dimension fractional maximal operator, and we calculate how the constant depends on $p$ and $\alpha$.

\begin{lemma}[Hedberg-type estimate]\label{lem:Hedberg}
Let $\Omega \subset \Rn$ be a bounded, open set.  
Let $p \in \PPln(\Omega)$ satisfy the inequalities $1 \le p^-\le p^+<\infty$.
Assume that  $\alpha:\Omega \to (0, n)$ satisfies   $\alpha^- >0$  and   $( \alpha p )^+ <n$.
If $\delta(x) := \frac{n-\alpha(x) p(x)}{p(x)}$, then 
there exists a constant $c$ such that
for every $\ve: \Omega \to(0, \infty)$, with $\ve(x) \le \alpha(x)$ for all $x \in \Omega$, the inequality
\[
I_{\alpha(\cdot)}f(x) \le c\,  \max\Big\{1, \frac{1}{\delta(x)}\Big\}^{\frac{p^+ -1}{p^+}} (\M_{\alpha(\cdot)-\ve(\cdot)} f(x))^{\frac{\delta(x)}{\delta(x) + \ve(x)}}
\]
holds
for all $f \in L^{\p}(\Omega)$ with $\|f\|_\p \le 1$. Here $c$ depends only on the dimension $n$, $\ve^-$,  the $\log$-Hölder constant of $p$, and 
$\diam(\Omega)$.
\end{lemma}

\begin{proof}
For a ball $B(x,r)$ we write
\[
\begin{split}
I_{\alpha(\cdot)}f(x)
&= \int_{B(x, r)\cap \Omega} \frac{|f(y)|}{|x-y|^{n- \alpha(x)}} \, dy + \int_{\Omega \setminus B(x, r)} \frac{|f(y)|}{|x-y|^{n- \alpha(x)}} \, dy\\
& =: I + II.
\end{split}
\]
For the first term we obtain
\[
\begin{split}
I &\le \sum_{j=0}^\infty \int_{B(x, r 2^{-j}) \setminus B(x, r 2^{-j-1}) \cap \Omega}\frac{|f(y)|}{|x-y|^{n- \alpha(x)}} \, dy\\
&\le c(n) \sum_{j=0}^\infty  \frac{(r 2^{-j})^n(r 2^{-j})^{\alpha(x) - \ve(x)}}{(r 2^{-j})^{\alpha(x) - \ve(x)}} \frac{1}{|B(x, r 2^{-j})|}\int_{B(x, r 2^{-j}) \cap \Omega}\frac{|f(y)|}{(r2^{-j-1})^{n- \alpha(x)}} \, dy\\
&\le c(n) r^{\ve(x)} 2^{n- \alpha(x)}\sum_{j=0}^\infty 2^{-j\ve(x)} \M_{\alpha(\cdot) -\ve(\cdot)}f(x) 
=  \frac{c(n)2^{\ve(x)}}{ 2^{\ve(x)}-1} r^{\ve(x)} \M_{\alpha(\cdot) -\ve(\cdot)}f(x).
\end{split}
\]
Next we estimate the second term.
Since $\|f\|_\p \le 1$  we use Lemma~\ref{lem:rieszvsM2} in order to obtain that
\[
\begin{split}
II &\le  c \;\max\Big\{1, \frac{p(x)}{n-\alpha(x) p(x)}\Big\}^{ \frac{1}{(p^+)'}} r^{-\frac{n-\alpha(x) p(x)}{p(x)}}
\\
&= c \; \max\Big\{1, \frac{1}{\delta(x)}\Big\}^{ \frac{1}{(p^+)'}} r^{-\delta(x)},
\end{split}
\]
and the constant depends only on $n$, the $\log$-Hölder constant of $p$, and 
$\diam(\Omega)$.

If 
\[
 \Big( \M_{\alpha(x)-\ve} f(x) \Big)^{-\frac{1}{\delta(x) + \ve(x)}}< \diam(\Omega)
\]
we choose
\[
r(x) = \Big( \M_{\alpha(x)-\ve} f(x) \Big)^{-\frac{1}{\delta(x) + \ve(x)}}.
\]
Hence,
\[
\begin{split}
I_{\alpha(\cdot)}f(x)&\le c \max\Big\{1,  \frac{1}{\delta(x)}\Big\}^{ \frac{1}{(p^+)'}}(\M_{\alpha(x)-\ve(x)} f(x))^{\frac{\delta(x)}{\delta(x) + \ve(x)}},
\end{split}
\]
where the constant depends only on the dimension $n$, the $\log$-Hölder constant of $p$, $\ve^-$, and 
$\diam(\Omega)$.

If 
\[
 \Big( \M_{\alpha(x)-\ve} f(x) \Big)^{-\frac{1}{\delta(x) + \ve(x)}}\ge \diam(\Omega)
\]
we choose $r(x) = \diam(\Omega)$.
Thus we obtain
\[
\begin{split}
I_{\alpha(\cdot)}f(x)&\le  I \le  c \diam(\Omega)^{\ve(x)} \M_{\alpha(\cdot) -\ve(\cdot)}f(x) \le c  \diam(\Omega)^n (\M_{\alpha(x)-\ve} f(x))^{\frac{\delta(x)}{\delta(x) + \ve(x)}},
\end{split}
\]
where the constant depends only on the dimension $n$, $\ve^-$, and  $\diam(\Omega)$.
\end{proof}

Next we prove our main Theorem~\ref{thm:Riesz-limit-case} which
is a  generalisation  of \cite[Theorem 1.2]{MarS21} to the variable order case.
We  clarify dependences of the final constants of the given parameters in Remark \ref{rmk:vakiot}.

\begin{proof}[Proof of Theorem~\ref{thm:Riesz-limit-case}]
Let $p \in \PPln(\Omega)$  satisfy the inequality  $p(x) < \frac{n}{\alpha(x)}$.
Then $\|f\|_{\p} \le 1$ by the assumption $\|f\|_{L^{\frac{n}{\alpha(\cdot)}}(\Omega)}\le \frac{1}{2(1 + |\Omega|)}$  and Corollary 3.3.4 of \cite{DieHHR11}.
By Hedberg's lemma, Lemma~\ref{lem:Hedberg}, we obtain
\[
I_{\alpha(\cdot)}f(x) \le c\, \max\Big\{1,  \frac{1}{\delta(x)}\Big\}^{ \frac1{(p^+)'}} (\M_{\alpha(\cdot)-\ve(\cdot)} f(x))^{\frac{\delta(x)}{\delta(x) + \ve(x)}},
\]
where $c$ depends only on the dimension $n$, $\ve^-$, the $\log$-Hölder constant of $p$, and 
$\diam(\Omega)$.
Let $1<r <  \min\{2, \frac{n}{\alpha^+}\}$.
By Hölder's inequality we obtain
\[
\M_{\alpha(\cdot) -\ve(\cdot)} f(x) \le \big(\M_{r(\alpha(\cdot) -\ve(\cdot))} |f|^r(x)\big)^{\frac1r}.
\]
Thus we have
\[
I_{\alpha(\cdot)}f(x) \le c\, \max\Big\{1,  \frac{1}{\delta(x)}\Big\}^{ \frac1{(p^+)'}} \big(\M_{r(\alpha(\cdot) -\ve(\cdot))} |f|^r(x)\big)^{\frac{\delta(x)}{r(\delta(x) + \ve(x))}}.
\]
These  estimates yield
\begin{equation}\label{equ:vakio-c-1}
\begin{split}
&\Ha_\infty^{n-r(\alpha(\cdot)- \epsilon(\cdot))}(\{x \in \Omega: I_{\alpha(\cdot)}f(x) >t\})\\
&\quad \le 
\Ha_\infty^{n-r(\alpha(\cdot)- \epsilon(\cdot))}\bigg(\Big\{x \in \Omega: \M_{r(\alpha(\cdot) -\ve(\cdot))}|f|^r (x) >
\Big(c\, \min\{1, \delta(x)\}^{ \frac1{(p^+)'}} \, t\bigg)^{\frac{r(\delta(x) + \ve(x))}{\delta(x)}}\Big\}\bigg),
\end{split}
\end{equation}
where $c$ depends only on the dimension $n$, $\ve^-$, and the $\log$-Hölder constant of $p$, and 
$\diam(\Omega)$.

Let us recall that $\delta(x)=\frac{n- \alpha(x) p(x)}{p(x)}$. Let $\sigma \in (0, 1)$ be a small number and choose
$p_\sigma(x) := \frac{n}{\alpha(x)}- \sigma$.  Note that  
 the function
$p_\sigma$ is $\log$-Hölder continuous provided that $\alpha$ is $\log$-Hölder continuous, and the $\log$-Hölder constant of $p_\sigma$ is independent of  $\sigma$. Moreover
 the values of the function
 $p_\sigma$ can be chosen to be near the critical value $\frac{n}{\alpha(x)}$. 
With this choice we have
\[
\delta (x) = \frac{n- \alpha(x) p_\sigma(x)}{p_\sigma(x)} = \frac{ \alpha(x)^2\sigma}{n- \alpha(x)\sigma} 
\in \Big[\frac{(\alpha^-)^2}{n} \sigma, \frac{ (\alpha^+)^2}{n- \alpha^+} \sigma   \Big].
\]
Hence $\delta(x) \to 0^+$ uniformly as $\sigma \to 0^+$.
Assume that $t_0$ is such that 
$c t_0 =e$, where $c$ is from \eqref{equ:vakio-c-1}. 
For every  $t>t_0$ 
we choose $\sigma$ so small that
\begin{equation}\label{equ:delta}
c\, \min\{1, \delta(x)\}^{ \frac1{(p_\sigma^+)'}} \, t \approx e,
\end{equation}
where we denote $A \approx B$ if is there exists a  positive constant $c$  independent of $A$ and $B$ such that 
$c^{-1}A \le B \le c A$.

Thus we have
\[
\begin{split}
&\Ha_\infty^{n-r(\alpha(\cdot)- \epsilon(\cdot))}(\{x \in \Omega: |I_{\alpha(\cdot)}f(x)| >t\})\\
&\quad \le 
\Ha_\infty^{n-r(\alpha(\cdot)- \epsilon(\cdot))}\Big(\Big\{x \in \Omega: \M_{r(\alpha(\cdot) -\ve(\cdot))}|f|^r (x) >
c \exp \Big({\frac{r(\delta(x) + \ve(x))}{\delta(x)}}\Big)\Big\}\Big)\\
&\quad \le \Ha_\infty^{n-r(\alpha(\cdot)- \epsilon(\cdot))}\Big(\Big\{x \in \Omega: \M_{r(\alpha(\cdot) -\ve(\cdot))}|f|^r (x)
 > c\exp \Big(\frac{r \ve^-}{\delta^+}\Big)\Big\}\Big).
\end{split}
\]
Hence by Lemma~\ref{lem:Haudorff-1} we obtain
\[
\Ha_\infty^{n-r(\alpha(\cdot)- \epsilon(\cdot))}(\{x \in \Omega: |I_{\alpha(\cdot)}f(x)| >t\})
\le c \exp \Big(-\frac{r\ve^-}{\delta^+}\Big). 
\]
For the exponent we obtain 
\[
\delta^+ = \sup_{x \in \Omega} \frac{n-\alpha(x) p_\sigma(x)}{p_\sigma(x)}
= \sup_{x \in \Omega} \frac{\alpha(x)^2\sigma}{n-\alpha(x) \sigma} = \frac{(\alpha^+)^2\sigma}{n-\alpha^+ \sigma},
\]
and thus
\[
\frac{\delta(x)}{\delta^+} = \frac{\alpha(x)^2\sigma}{n-\alpha(x) \sigma} \cdot \frac{n-\alpha^+ \sigma}{(\alpha^+)^2\sigma}
\in \Big[\frac{(\alpha^-)^2(n- \alpha^+)}{n(\alpha^+)^2}, \frac{n }{n-\alpha^+} \Big].
\]
This yields that $\delta(x) \approx \delta^+$ for every $x \in \Omega$  and hence by \eqref{equ:delta} we have
$\frac{1}{\delta^+} \approx \frac{1}{\delta(x)} \approx  t^{(p_\sigma^+)'} = t^{\frac{n-\sigma \alpha^-}{n-(1+\sigma)\alpha^-}}$. This implies that
\[
\Ha_\infty^{n-r(\alpha(\cdot)- \epsilon(\cdot))}(\{x \in \Omega: I_{\alpha(\cdot)}f(x) >t\})
\le c_1 \exp \Big({-c_2 r \ve^- t^{\frac{n-\sigma \alpha^-}{n-(1+\sigma)\alpha^-}}}\Big), 
\]
 where $c_1$  depend only on  $n$, $\alpha^-$, $\alpha^+$, and $\log$-Hölder constant of  $\alpha$, and $\diam(\Omega)$, and 
$c_2$ depend only on  $n$, $\alpha^-$, $\alpha^+$, and $\log$-Hölder constant of  $\alpha$.
Finally we choose $\ve(x) := (r-1) \alpha(x)$ and obtain
\[
\Ha_\infty^{n-\alpha(\cdot)}(\{x \in \Omega: I_{\alpha(\cdot)}f(x) >t\})
\le c_1 \exp \Big({-c_2 r(r-1) \alpha^- t^{\frac{n-\sigma \alpha^-}{n-(1+\sigma)\alpha^-}}}\Big), 
\]
whenever $t>t_0$.
Since the left-hand side and the constants are independent of $\sigma$, we take $\sigma \to 0^+$ and obtain 
\[
\Ha_\infty^{n-\alpha(\cdot)}(\{x \in \Omega: I_{\alpha(\cdot)}f(x) >t\})
\le c_1 \exp \Big({-c_2 r(r-1) \alpha^- t^{\frac{n}{n-\alpha^-}}}\Big).
\]

The claim holds also if $0<t\le t_0$. Indeed, for $0<t\le t_0$ 
\[
\begin{split}
\Ha_\infty^{n-\alpha(\cdot)}(\{x \in \Omega: I_{\alpha(\cdot)}f(x) >t\})
&\le \Ha_\infty^{n-\alpha(\cdot)}(\Omega)\\
&\le    \Ha_\infty^{n-\alpha(\cdot)}(\Omega) \exp\Big({c t^{\frac{n}{n-\alpha^-}}}\Big) \exp \Big({-ct^{\frac{n}{n-\alpha^-}}}\Big)\\
&\le   \Ha_\infty^{n-\alpha(\cdot)}(\Omega) \exp \Big({c t_0^{\frac{n}{n-\alpha^-}}}\Big) \exp \Big({-c t^{\frac{n}{n-\alpha^-}}}\Big), 
\end{split}
\]
where  $\Ha_\infty^{n-\alpha(\cdot)}(\Omega) \le (1+\diam(\Omega))^n$ by the definition of the Hausdorff content.
Hence, the theorem is proved.
\end{proof}

\begin{remark}\label{rmk:vakiot}
The estimates in the previous proof yield Theorem~\ref{thm:Riesz-limit-case} with a constant $c_1$ which  depends only on $\diam(\Omega)$, $n$, 
$\alpha^-$, $\alpha^+$, and $\log$-Hölder constant of $\alpha$, and
a constant $c_2$  which depends only on  $n$, $\alpha^-$, $\alpha^+$, and $\log$-Hölder constant of  $\alpha$.
\end{remark}

\section{Applications to non-smooth domains}
\label{John}

The definition of a bounded John domain goes back to F. John
\cite[Definition, p. 402]{J} who defined an inner radius and an outer radius domain, and later this domain was renamed as a John domain
in  \cite[2.1]{MS79}. We generalise this definition so that the shape of the   John cusp can depend on the point.
If $s$ is a constant function we have a so called $s$-John domain studied in \cite{SmiS90}. For other studies and generalisations of John domains  we refer to \cite{HH-S2, HH-SK}. 

\begin{definition}\label{bounded-john}
Let $D\subset \Rn$, $n\geq 2$, be a bounded domain, and $s: D \to [1,\infty)$ a function. The domain $D$ is an $s(\cdot)$-John domain if there exist constants $0< \alpha \le \beta<\infty$ and a point $x_0 \in D$ such that each point $x\in D$ can be joined to $x_0$ by a rectifiable curve $\gamma_x:[0,\ell(\gamma_x)] \to D$, parametrized by its arc length, such that $\gamma_x(0) = x$, $\gamma_x(\ell(\gamma_x)) = x_0$, $\ell(\gamma_x)\leq \beta\,,$ and
\[
t^{s(x)} \leq  \frac{\beta}{\alpha} \dist\big(\gamma_x(t), \partial D\big) \quad \text{for all} \quad t\in[0, \ell(\gamma_x)].
\]
The point $x_0$ is called a John center of $D$ and 
$\gamma_x$ is called a John curve of $x$.
\end{definition}

\begin{example}
We construct a mushrooms-type domain. Let $(r_m)$ be a decreasing sequence of positive real numbers converging to zero. 
Let $Q_{m}$, $m=1,2,\dots$, be a closed cube in $\Rn$ with side length $2r_{m}$. 
 Let $\phi:[0, \infty) \to [0. \infty)$ be an increasing function with $\lim_{t \to 0^+} \phi(t) = \phi(0) =0$ and $\phi(t) >0$ for $t>0$.
Let $P_{m}$, $m=1,2,\dots$,
be a closed rectangle in $\Rn$ which has side length $r_{m}$ for one side and $2\varphi (r_{m})$ for the remaining $n-1$ sides.  Let $Q:=[0, 12] \times [0, 12]$. 
We attach $Q_{m}$ and $P_{m}$ together creating 'mushrooms' which we then attach, as pairwise disjoint sets, to the side $\{(0, x_2, \ldots, x_n): x_2, \ldots, x_n >0\}$ of $Q$ so that the distance from the mushroom to the origin is at least $1$ and at most $4$, see Figure~\ref{fig:domain}.  
We have to assume here also that $\varphi(r_m) \le r_m$.
We need copies of the mushrooms.  By an isometric mapping we transform these mushrooms onto the side $\{(x_1, 0, \ldots, x_n): x_1, x_3, \ldots, x_n >0\}$ of $Q$  and
denote them by $Q_m^*$ and $P_m^*$. So again the distance from the mushroom to the origin is at least $1$ and at most $4$.
We define
\begin{eqnarray}\label{eq:mushroon-domain}
D:=\textrm{int}\left(Q \cup\bigcup_{m=1}^{\infty}\Big(Q_{m}\cup P_{m}\cup Q^{*}_{m}\cup P^{*}_{m}\Big)\right).
\end{eqnarray}
See Figure~\ref{fig:domain}. 

We set that $\phi(t) := t^{\frac32}$.
We define that $s(x) :=1$ in $Q$, $s(x) := \frac32$ in  $\bigcup_{m=1}^{\infty} \Big(Q_{m} \cup Q_{m}^*\big)$, and grows lineary from $1$ to $\frac32$ in each $P_m$ and each $P_m^*$.
Then $D$ is an $s(\cdot)$-John domain, which can be seen as in \cite[Lemma 6.2]{HH-SK}.
\end{example}

\begin{figure}[ht!]
\includegraphics[width=11 cm]{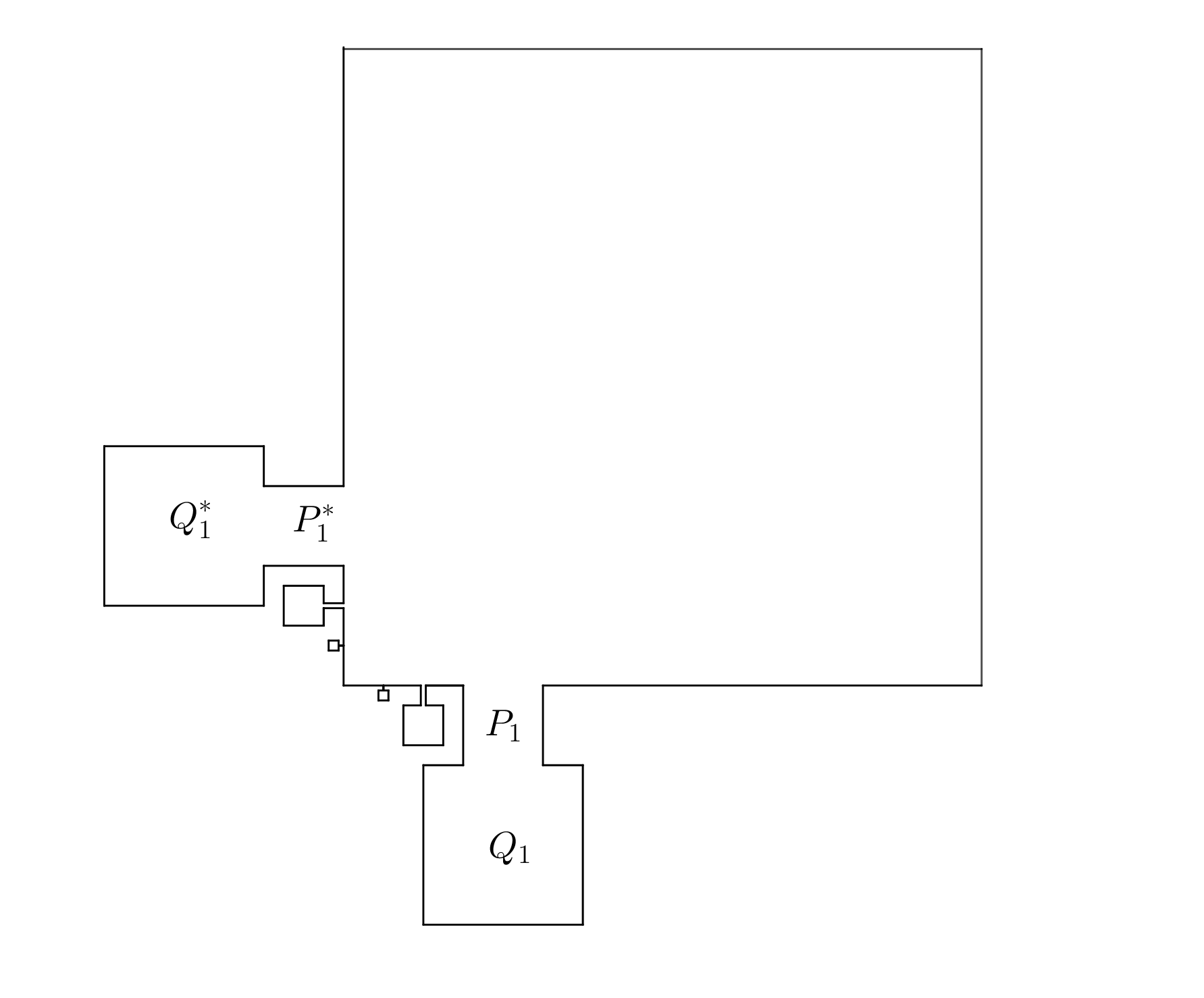}
\caption{$s(\cdot)$-John domain.}\label{fig:domain}
\end{figure}

Next we prove a chaining results for $\s$-John domains.
We refer to \cite[Theorem 9.3]{Hajlasz-Koskela} for the proof in the classical case, and  \cite[Lemma 3.5]{HH-SK} and \cite[4.3. Lemma]{HH-S1} for  generalisations.

\begin{lemma}\label{lem:pallot} 
 Let $D \subset \Rn\,, n\geq 2\,,$ be a $s(\cdot)$-John domain 
with  John constants $\alpha$ and $\beta$. Let  $x_0 \in D$ the John center. 
Then for every $x\in D\setminus B(x_0, \dist(x_0, \partial D))$ there exists a sequence of balls $\big(B(x_i, r_i)\big)$ 
such that $B(x_i, 2r_i)$ is in $D\,$ for each
$i=0,1,\dots\,,$ and
for some constants $K=K(\alpha, \beta, \dist(x_0, \partial D))$, $N=N(n)$, and $M=M(n)$ 
\begin{itemize}
\item[(1)]
$B_0 = B\Big(x_0, \frac12 \dist(x_0,  \partial D)\Big)$;
\item[(2)]
$\dist(x, B_i)^{s(x)}\leq K r_i$, and $r_i \to 0 $ as $i\to \infty$;
\item[(3)]
no point of the domain $D$ belongs to more than $N$ balls $B(x_i, r_i)$; and
\item[(4)]
$|B(x_i, r_i) \cup B(x_{i+1}, r_{i+1})| \leq M |B(x_i, r_i) \cap B(x_{i+1}, r_{i+1})|$.
\end{itemize}
\end{lemma}

\begin{proof}
Let $\gamma$ be a John curve joining $x$ to $x_0$. Let us write 
\[
B'_0 := B\Big(x_0, \frac14 \dist\big(x_0, \partial D \big)\Big).
\]
Let us consider the balls $B'_0$ and
\[
B \Big(\gamma(t), \frac14 \dist\big(\gamma(t), \partial D \cup \{x\}\big)\Big),
\]
when $t\in(0,l)$,
here $l$ stands for the length of $\gamma$. 
By the Besicovitch covering theorem,  there is a sequence of closed balls $\overline{B'_1}$, $\overline{B'_2}$, \ldots and $\overline{B'_0}$ 
that cover $\{\gamma(t): t\in[0, l] \}\setminus \{x\}$ and have a uniformly bounded overlap depending on $n$ only, \cite[2.7]{Mattila}.
 Let us define open balls $B_i := 2 B'_i$ with center at $x_i:= \gamma(t_i)$ and radius 
 $r_i := \frac12 \dist\big(x_i, \partial D \cup \{x\}\big)$\,, $i=1\,, 2\,, \ldots\,.$ 
That is $B_i=B(x_i,r_i)\,,$
$i=1\,, 2\,, \ldots\,.$ 

For a ball $B_0:=B(x_0,\frac{1}{2}\dist (x_0,\partial D))$ we obtain by the definition of $s(\cdot)$-John domain
\begin{equation*}
\dist(x, B_0)^{s(x)} \le  \ell(\gamma_x)^{s(x)}
\le \frac{\beta}{\alpha}\dist (x_0,\partial D)= \frac{\beta r_0}{2 \alpha}. 
\end{equation*}
Assume then that $i\ge 1$.
If $r_i = \frac12 \dist(x_i, x)$, then 
$ \dist(x,B_i) \le 2 r_i$. 
If $r_i = \frac12 \dist(x_i, \partial D)$, then the definition of a $s(\cdot)$-John domain gives that
\[
\begin{split}
\dist(x, B_i)^{s(x)} &\leq \dist(x, x_i)^{s(x)}\\ 
&\leq t_i^{s(x)} \leq \frac{\beta}{\alpha} \dist(\gamma(t_i), \partial D) \leq \frac{\beta r_i}{2 \alpha}.
\end{split}
\]
Thus, the first part of property (2) holds.

Let us renumerate the balls. Let $B_0$ be as before. If we have chosen balls $B_i$, $i=0, 1, \ldots, m$, then we choose that the ball $B_{m+1}$ is the ball for which $x_j \in B_m$ and $t_j < t_m$, by remembering that $\gamma(t_j) = x_j$ and $\gamma(t_m) = x_m$. Hence, $r_i \to 0$ and $x_i\to x$, as $i\to \infty$. Thus, the second part of property (2) holds.

The point $x$ does not belong to any ball. Let $x'$ be any other point in the domain $D$. The point $x'$ cannot belong to the balls $B_i$ with $3r_i <\dist(x', x)$. If $x' \in B_i$, then 
\begin{equation*}
2 r_i \leq \dist(x, x_i) \leq \dist(x, x') + r_i.
\end{equation*}
Thus, we obtain that $x' \in B_i$ if and only if 
$$\frac13 \dist(x', x) \leq r_i \leq \dist(x, x').$$
The Besicovitch covering theorem implies that the balls with radius of $\frac14$ of the original balls are disjoint. Hence $x'$ belongs to less than or equal to
\[
N \frac{\vert B(x',2r)\vert}{\vert B(0,\frac1{12} r)\vert} = 24^n N
\]
balls $B_i$, where the constant $N$ is from the Besicovitch covering theorem and depends on the dimension $n$ only.
Hence, property (3) holds.

If $r_i = \frac12 \dist(x_i, \partial D)$ and $r_{i+1} = \frac12 \dist(x_{i+1}, \partial D)$, then $r_{i+1} \geq \frac12 r_i$ (since $x_{i+1} \in B_i$) and thus we obtain
\[
\frac{|B_i|}{|B_{i+1}|} \leq \left ( \frac{r_i}{\frac12 r_i}\right)^n = 2^n. 
\]
If $r_i = \frac12 \dist(x_i, x)$ and $r_{i+1} = \frac12 \dist(x_{i+1}, x)$, then $r_{i+1} \geq \frac12 r_i$ and thus we obtain
$|B_i|/|B_{i+1}| \leq 2^n$.   
If $r_i = \frac12 \dist(x_i, \partial D)$ and $r_{i+1} = \frac12 \dist(x_{i+1}, x)$, then  $r_i \leq \frac12 \dist(x_i, x)$ and we obtain the same ratio as before.
Similarly in the case when $r_i = \frac12 \dist(x_i, x)$ and  $r_{i+1} = \frac12 \dist(x_{i+1}, \partial D)$. We have shown
$|B_i| \leq 2^n |B_{i+1}|$. 
In the same manner we obtain $2r_{i+1} \le 3 r_i$ and hence $2^n|B_i| \geq  |B_{i+1}|$.  These yield property (4).
\end{proof}

The previous lemma yields the following lemma.

\begin{lemma}\label{lem:p-w-Riesz}
 Let $D\subset \Rn\,, n\geq 2\,,$ be an $s(\cdot)$-John domain. Then
 \[
 \big|u(x) - u_{B(x_0, \dist(x_0, \partial D))} \big| \le c  \It_{s(\cdot)}|\nabla u| (x)
 \] 
 for all $u \in L^{1}_1(D)$. Here the constant $c$ depends only on $n$, $\alpha$, $\beta$, $s^+$, and $\dist(x_0, \partial D)$.
\end{lemma}

\begin{proof}
If $x\in  B(x_0, \dist(x_0, \partial D))$, then 
\[
\big|u(x) - u_{B(x_0,  \dist(x_0, \partial D))}\big| \leq \frac{\diam (B(x_0,  \dist(x_0, \partial D)))^n}{
n|B(x_0,  \dist(x_0, \partial D))|} \int_{
 B(x_0,  \dist(x_0, \partial D))}  \frac{|\nabla u (y)|}{ |x-y|^{n-1}} \, dy\,
\]
by \cite[Lemma 7.16]{Gilbarg-Trudinger}. Since the domain is bounded we have 
\[
\begin{split}
|x-y|^{s(x)(n-1)}  &< \diam(D)^{(s(x)-1)(n-1)} |x-y|^{n-1}\\ 
&\le (1+\diam(D))^{(\sup s-1)(n-1)} |x-y|^{n-1}\\ 
&\le (1+2\beta)^{(\sup s-1)(n-1)} |x-y|^{n-1}, 
\end{split}
\] 
which gives the claim.

Let us then assume that  $x\in D\setminus B(x_0, \dist(x_0, \partial D))$ and let 
$(B_i)_{i=0}^{\infty}$ be a sequence of balls constructed in Lemma~\ref{lem:pallot}. Property (2) gives that $\dist(x, B_i) \to 0$ as $i \to \infty$. Thus, property (2)  and the Lebesgue differentiation theorem \cite[Section 1, Corollary 1]{Stein} imply that 
$u_{B_i} \to u(x)$ when  $i\to\infty$
 for almost every $x$. We obtain
\[
\begin{split}
|u(x) - u_{B_0}| &\leq \sum_{i=0}^\infty |u_{B_i} - u_{B_{i+1}}| \\
&\leq \sum_{i=0}^\infty \left (| u_{B_i} - u_{B_i \cap B_{i+1}}|  + |u_{B_{i+1}} - u_{B_i \cap B_{i+1}}| \right)\\
&\leq  \sum_{i=0}^\infty \left (\vint_{B_i \cap B_{i+1}}  |u(y)- u_{B_i} | \, dy  + \vint_{B_i \cap B_{i+1}} |u(y) - u_{B_{i+1}}|\, dy \right)\,.\\
\end{split}
\]
By property (4) 
\[
|u(x) - u_{B_0}| 
\leq  2C\sum_{i=0}^\infty \vint_{B_i}  |u(y)- u_{B_i} | \, dy. 
\]
Using the $(1,1)$-Poincar\'e inequality in a ball $B_i$, \cite[Section 7.8]{Gilbarg-Trudinger}, we obtain
\[
\begin{split}
|u(x) - u_{B_0}|
&\leq  C\sum_{i=0}^\infty r_i \vint_{B_i}  |\nabla u(y)| \, dy. \\
\end{split}
\]

Thus, for each $z \in B_i$ we obtain by property (2) that
\[
|x-z| \le \dist(x, B_i) + 2r_i \le (C r_i)^{\frac1{s(x)}} + 2r_i
\le  C r_i^{\frac1{s(x)}},
\]
where in the last inequality we used that $D$ is bounded.
Hence, we have $ C |x-z|^{s(x)} \le  r_i$.
Using this we obtain by property (3) that
\[
\begin{split}
|u(x) - u_{B_0}|
&\leq  C\sum_{i=0}^\infty r_i \vint_{B_i}  |\nabla u(y)| \, dy 
\leq C\sum_{i=0}^\infty \int_{B_i}  \frac{|\nabla u(y)|}{r_i^{n-1}} \, dy \\
&\leq  C\sum_{i=0}^\infty  \int_{B_i}  \frac{|\nabla u(y)|}{|x-y|^{s(x)(n-1)}} \, dy
\leq  C  \int_{D}  \frac{|\nabla u (y)|}{|x-y|^{s(x)(n-1)}} \, dy. \qedhere
\end{split}
\]
\end{proof}

Now Lemma~\ref{lem:Samko} and Remark~\ref{RieszNewNotation} yield  the following theorem, where we do not need continuity of $s$.
Note that $\alpha^->0$ if and only if $s^+< \frac{n}{n-1}$ and $(\alpha p)^+ =(n-s^-(n-1))p <n$ if and only if $p < \frac{n}{n-s^-(n-1)}$. 

 \begin{theorem}[Sobolev-Poincar\'e inequality in the case $p>1$]\label{thm:main-p>1}
 Let $D\subset \Rn\,, n\geq 2\,,$ be an  $s(\cdot)$-John domain. Assume that 
 \begin{equation}\label{equ:rajat-s}
 1 \le s^-  \le s^+ < \frac{n}{n-1} \quad \text{ and } \quad 1<p<\frac{n}{n-s^-(n-1)}.
 \end{equation} 
 Then there exists a constant $c$ such that 
 \[
 \|u - u_B\|_{L^{q(\cdot)}(D)} \le c \|\nabla u\|_{L^p(D)},
 \]
 for all $u\in L^1_p(D)$.
 Here $q(x) := \frac{np}{pn(s(x)-1)+n -s(x) p}$
 and $B:= B(x_0, \dist(x_0, \partial D))$.  The constant $c$ depends only on the dimension $n$, $\alpha$, $\beta$, $s^+$, $\dist(x_0, \partial D)$ and $\log$-Hölder constant of $p$.
  \end{theorem}

\begin{proof}
By Lemma~\ref{lem:p-w-Riesz} we obtain
\[
 \big|u(x) - u_{B(x_0, \dist(x_0, \partial D)} \big| \le c \It_{s(\cdot)} |\nabla u|(x).
 \] 
 Note that \eqref{equ:rajat-s} yields assumption of Lemma~\ref{lem:Samko} by Remark~\ref{RieszNewNotation}.
Thus  we obtain  that
 \[
 \|u - u_{B(x_0, \dist(x_0, \partial D)}\|_{L^{q(\cdot)}(D)} \le c \|\nabla u\|_{L^p(D)}. \qedhere
 \]
\end{proof}

In the case $p=1$ we use Theorem~\ref{hteJ} and hence we need to assume that $s$ is $\log$-Hölder continuous.

\begin{theorem} [Sobolev-Poincar\'e inequality in the case $p=1$] \label{thm:main-p=1}
 Let $D\subset \Rn\,, n\geq 2\,,$ be an $s(\cdot)$-John domain. Assume that $s \in \PPln(\Omega)$ satisfies 
 $1\le s^-  \le s^+ < \frac{n}{n-1}$. Then there exists a constant $c$ such that 
 \[
 \|u - u_B\|_{L^{q(\cdot)}(D)} \le c \|\nabla u\|_{L^1(D)},
 \]
for all $u \in L^1_1(D)$, where $q(x) := \frac{n}{s(x)(n-1)}$,
 and $B:= B(x_0, \dist(x_0, \partial D))$. 
 The constant $c$ depends only on the dimension $n$, $\alpha$, $\beta$, $s^+$, $\dist(x_0, \partial D)$, and the $\log$-Hölder constant of $\alpha$.
\end{theorem}

\begin{proof}
 Let us first assume that $\|\nabla u\|_{1} \leqslant 1$. We
  show that $\varrho_{L^{q(\cdot)}(D)}(|u-u_B|)$
  is uniformly bounded.  For every  $j \in \Z$ we set 
 \[
 D_j :=\{x\in D: 2^j < |u(x)-u_B| \leqslant 2^{j+1}\}
 \text{ and } v_j
  := \max\big\{0, \min\{|u- u_B| -2^j, 2^j\}\big\}.
  \]
   From 
  the lattice property of Sobolev functions it follows ${v_j\in
  L^{1}_1(D)}$. By Lemma~\ref{lem:p-w-Riesz} we have
  \begin{equation*}
    |v_j(x)- (v_j)_{B}| 
    \leqslant  c\, \It_{s(\cdot)} |\nabla v_j| (x)
  \end{equation*}
  for almost every $x\in D$. 
  
   We obtain by the pointwise inequality $v_j
  \leq |u-u_B|$ and by the  Poincar{\'e} inequality in a ball
  that
  \begin{align}\label{S-PineqW:equ2}
    \begin{aligned}
    \hspace*{-4mm}  v_j(x) & \leqslant |v_j(x)-(v_j)_{B}| + (v_j)_{B} \leqslant c \,
      \It_\s |\nabla v_j| (x) + \vint_B |u-u_B| \,dx
      \\
      &\leqslant c\, \It_\s |\nabla v_j| (x)+ c\!  \diam(B) \vint_B  |\nabla u
      |\,dx
      \\
      &\leqslant c \, \It_\s |\nabla v_j| (x)+ c \frac{\diam(B)}{|B|} \leqslant c_1 \,(\It_\s |\nabla v_j| (x)+ 1),
    \end{aligned}
  \end{align}
  where in the second to last inequality we used that $\|\nabla u\|_{1} \leqslant 1$, and in the last inequality we note that $\diam(B)$ and $|B|$ both depends only on the distance from the John center to the boundary.
  For the rest of this proof we fix the constant $c_1$ to denote the
  constant on the last line.  It depends only on $n$,  $\alpha$, $\beta$, $s^+$ and $\dist(x_0, \partial D)$.

  Using the definition of $D_j$ we get
  \begin{align*}
    \int_D |u(x)-u_B|^{q(x)} \,dx &= \sum_{j=-\infty}^\infty~ \int_{D_j} |u(x)-u_B|^{q(x)} \,dx
\leqslant \sum_{j=-\infty}^\infty~ \int_{D_j} 2^{(j+1)q(x)}
    \,dx\,.
  \end{align*}
  For every $x\in D_{j+1}$ we have $v_j (x)=2^{j}$ and thus
  obtain by \eqref{S-PineqW:equ2} the pointwise inequality $c_1 \It_\s
  |\nabla v_j| (x)+ c_1 > 2^{j}$ for almost every $x\in D_{j+1}$.
  Note that if $a+b>c$, then $a>\frac12 c$ or $b> \frac 12 c$. Thus 
  \begin{equation*}
    \begin{aligned}
      &\sum_{j=-\infty}^\infty~ \int_{D_j} 2^{(j+1)q(x)} \,dx 
\leqslant \sum_{j=-\infty}^\infty~ \int_{\{x\in D_j \colon
        c_1 \It_\s |\nabla v_j| (x)+ c_1> 2^{j-1}\}}
      \hspace*{-10mm}2^{(j+1)q(x)} \,dx
      \\
      &\quad \leqslant \sum_{j=-\infty}^\infty~ \int_{\{x\in D \colon
        c_1 \It_\s |\nabla v_j| (x)>2^{j-2}\}}\hspace*{-6mm}
      2^{(j+1)q(x)} \,dx + \sum_{j=-\infty}^\infty~ \int_{\{x\in
        D_{j} \colon c_1 > 2^{j-2}\}}\hspace*{-6mm} 2^{(j+1)
        q(x)} \,dx.
    \end{aligned}
  \end{equation*}
  Since $\| \nabla u\Vert_{1} \leqslant 1$, we
  obtain by Theorem~\ref{hteJ} and Remark~\ref{RieszNewNotation} for the first term on the right-hand
  side that
  \begin{equation*}
    \begin{aligned}
      \sum_{j=-\infty}^\infty~ &\int_{\{x\in D \colon c_1 \It_\s |\nabla
        v_j(y)| (x)> 2^{j-2}\}} 2^{(j+1)q(x)} \,dx
      \\
      &\le 2^{3 n}\sum_{j=-\infty}^\infty~ \int_{\{x\in D \colon c_1 \It_\s |\nabla
        v_j(y)| (x)> 2^{j-2}\}} 2^{(j-2)q(x)} \,dx
      \\
      &\leqslant c\, \sum_{j=-\infty}^\infty\bigg( \int_{D} |\nabla
      v_j(y)| \,dy + \big|\{ 0<|\nabla v_j| \leq 1\}\big| \bigg)
      \\
      &\leqslant c \,\sum_{j=-\infty}^\infty\bigg ( \int_{D_j}
      |\nabla u| \,dy + |D_j| \bigg)
= c \,\int_{D} |\nabla u| \,dy + c\,
      |D|.
    \end{aligned}
  \end{equation*}
  Let $j_0$ be the largest integer satisfying $c_1 > 2^{j_0-2}$. Hence
  \[
  \sum_{j=-\infty}^\infty \int_{\{x\in D_{j}: c_1 > 2^{j-2}\}} 2^{(j+1)
    q(x)} \,dx \leq \int_D \sum_{j=-\infty}^{j_0} 2^{(j+1)
    q(x)} \,dx \le c \, |D|.
  \]
  
  Then we conclude the proof by the scaling argument: Since $\|u- u_B\|_{q(\cdot)} \le c$ for all $\|\nabla u\|_1\le 1$, 
  we obtain the claim by applying this to $u/\|\nabla u \|_1$. 
\end{proof}

\begin{remark}
\begin{enumerate}
\item  If  $s\equiv 1$ is chosen, then  by Theorems \ref{thm:main-p>1} and \ref{thm:main-p=1} the classical Sobolev-Poincar\'e inequality is recovered  for $1$-John domains, \cite{Bojarski1988}.
\item Let  $s$ be a constant function and $1< s<n/(n-1)$.  The target spaces $L^{\frac{n}{s(n-1)}}(D)$ in Theorem \ref{thm:main-p=1} 
is optimal, while the target space $L^{\frac{np}{n - sp + pn(s-1)}}(D)$ in Theorem \ref{thm:main-p>1} is not the best possible, see \cite{HajK98, KilM00}. 
\item  Let  $s$ be a constant function and $1\le s<n/(n-1)$,  then the classical $(1,1)$-Poincar\'e inequality in an $s$-John domain is recovered.
And this yields the $(p,p)$-Poincar\'e inequality for all $1< p <\infty$. Recall that
it has been known that the $(p,p)$-Poincar\'e inequality holds for all $p\in [1,\infty )$ whenever
$1\le s < n/(n-1)$, \cite{SmiS90}.
\end{enumerate}
\end{remark}

 Let us recall the following result,  Theorem \ref{Israel} for bounded John domains.
 The result was proved for domains with a fixed cone condition by N. S. Trudinger \cite{Tru1967}.
 Domains with a fixed cone condition are examples of John domains. But John domains form a strictly larger class of domains than domains with a fixed cone condition.
 
 \begin{theorem}\cite{EH-S2001}\label{Israel}
  Let $D$ be a  $1$-John domain in $\R^n$, $n\geq 2$.
  There exists a constant $a>0$ such that
 \[
\int_D \exp\bigg( a\frac{ |u-u_D|}{
\vert\vert \nabla u\vert\vert _{L^n(D)}^n}\bigg)
^{\frac{n}{n-1}}
\,dx <\infty
\]
 for all $u\in L^1_n(D)$.
 \end{theorem}

Next  we  improve Theorem \ref{Israel}.
 In the next theorem the Hausdorff content is sharper than the Lebesgue measure in the following sense:   the equation  $|\{ x \in D : a |u-u_B|^{\frac{n}{s(n-1)}} >t \}|=0$ does not imply  
  the equation 
 $\Ha^{s(n-1)}_\infty\Big(\{ x \in D : a |u-u_B|^{\frac{n}{s(n-1)}} >t \} \Big)=0$, but the latter equation  implies the former one.

 When the integration is with respect to the Hausdorff content, the integration  is taken as a Choque integral, \cite{Ad1998}.
Let us define that 
\begin{equation}\label{IntegralDef}
\int_D |u| \, d \Ha^{n-\alpha(\cdot)}_\infty := \int_0^\infty \Ha^{n-\alpha(\cdot)}_\infty\big(\{x \in D : |u(x)|>t\}\big) \, dt. 
\end{equation}
Note  that $\Ha^{n-\alpha(\cdot)}_\infty$ is monotone by Lemma~\ref{lem:H} i.e. if $A\subset B$ then $\Ha^{n-\alpha(\cdot)}_\infty(A) \le \Ha^{n-\alpha(\cdot)}_\infty(B)$. Hence the function $t \mapsto \Ha^{n-\alpha(\cdot)}_\infty\big(\{x \in D : |u(x)|>t\}\big)$ is a decreasing function $\R \to [0, \infty)$ for every  $u:D \to \R$. By the decreasing property the function $t \mapsto \Ha^{n-\alpha(\cdot)}_\infty\big(\{x \in D : |u(x)|>t\}\big)$ is measurable. Thus, $\int_0^\infty \Ha^{n-\alpha(\cdot)}_\infty\big(\{x \in D : |u(x)|>t\}\big) \, dt$ is well-defined as a Lebesgue integral.
 For the notion of Choquet integral in applications we refer to \cite{Ad1998}.

\begin{theorem} [Sobolev-Poincar\'e inequality in the limit case ] \label{thm:main-p=n}
 Let $D$ be an  $s(\cdot)$-John domain in $\R^n$, $n\geq 2$.  Assume that $s \in \PPln(D)$ satisfies 
 $1\le s^-  \le s^+ < \frac{n}{n-1}$. When  $\alpha(x) := n-s(x)(n-1)$, then there exist positive constants $a$ and $b>0$ such that
 \[
\int_D \exp\big( a |u-u_B|^{\frac{n}{s^+(n-1)}} \big) \, d \Ha^{s(\cdot)(n-1)}_\infty \le b
\]
 for all $u \in L^{1}_{n/\alpha(\cdot)}(D)$ with $\|\nabla u\|_{L^{\frac{n}{\alpha(\cdot)}}(D)} \le 1$, here 
 $B:= B(x_0, \dist(x_0, \partial D))$.
\end{theorem}
 
 Theorem \ref{thm:main-p=n} yields Corollary~\ref{cor:main} as a special case when $s$ is a  constant function.
 
 \begin{proof}[Proof of Theorem \ref{thm:main-p=n}]
Let $m >0$ be a constant that  we will fix later.  
 By Lemma~\ref{lem:p-w-Riesz} there exists a constant $c_1$ such that
 \[
 \begin{split}
 &\int_D \exp\bigg(\bigg (\frac{m}{2(1+|D|)} \vert u - u_B\vert  \bigg)^{\frac{n}{s^+(n-1)}}\bigg) \, d \Ha^{s(\cdot)(n-1)}_\infty \\
 &\quad\le   \int_D \exp\bigg( \bigg(\frac{m c_1}{2(1+|D| )} \tilde I_{\s} |\nabla u| \bigg)^{\frac{n}{s^+(n-1)}} \bigg) \, d \Ha^{s(\cdot)(n-1)}_\infty \,.
 \end{split}
 \]
 By using the definition \eqref{IntegralDef} and splitting the interval of integration to two parts we obtain
 \[
 \begin{split}
 &\int_D \exp\bigg(\bigg (\frac{m}{2(1+|D|)} \vert u - u_B\vert  \bigg)^{\frac{n}{s^+(n-1)}}\bigg) \, d \Ha^{s(\cdot)(n-1)}_\infty \\
 &\quad\le   \int_0^\infty \Ha^{s(\cdot)(n-1)}_\infty\Bigg(\bigg\{x \in D : \exp\bigg( \bigg(\frac{m c_1}{2(1+|D|)} \tilde I_{\s} |\nabla u|(x) \bigg)^{\frac{n}{s^+(n-1)}} >t \bigg\} \Bigg) \, dt\\
 &\quad\le \int_0^1 \Ha^{s(\cdot)(n-1)}_\infty\Bigg(\bigg\{x \in D : \bigg(\frac{\tilde I_{\s} |\nabla u|}{2(1+|D|)|} \bigg)^{\frac{n}{s^+(n-1)}}> \frac{\log (t)}{(mc_1)^{n/(s^+(n-1))}}\bigg\}\Bigg) \, dt\\
&\qquad +  \int_1^\infty \Ha^{s(\cdot)(n-1)}_\infty \Bigg(\bigg\{x \in D : \bigg(\frac{\tilde I_{\s} |\nabla u|}{2(1+|D|)|}\bigg)^{\frac{n}{s^+(n-1)}} > \frac{\log (t)}{(mc_1)^{{n}/(s^+(n-1))}}\bigg\}\Bigg)\, dt\,.
\end{split}
\]
The integral over the  unit interval is estimated by
$ \Ha^{n-\alpha(\cdot)}_\infty(D)$.
We estimate the second integral over the unbounded interval. 
Let us first note that $\Big\| \frac{ |\nabla u|}{2(1+| D|)} \Big\|_{L^{\frac{n}{\alpha(\cdot)}}(D)} \le \frac{1}{2(1+|D|)}$.
By  Remark~\ref{RieszNewNotation} we have  $\tilde I_{\s}= I_{\alpha(\cdot)}$ where $\alpha(x) = n- s(x) (n-1)$.  The condition $\alpha^->0$ holds if and only if $s^+< \frac{n}{n-1}$ and
$\alpha^+ \le 1<n$  since $s^- \ge 1$.
Moreover $\frac{n}{n-\alpha^-}= \frac{n}{s^+(n-1)}$.
By Theorem~\ref{thm:Riesz-limit-case}  there exist constants $c_2$ and $c_3$ such that
\[
\begin{split}
 &\Ha^{n-\alpha(\cdot)}_\infty \Bigg(\bigg\{x \in D : \tilde I_{\s} \bigg(\frac{ |\nabla u|}{2(1+|D|)}\bigg) > \bigg(\frac{\log (t)}{(mc_1)^{n/(s^+(n-1))}} \bigg)^{\frac{s^+(n-1)}{n}}\bigg\}\Bigg)\\
  &\qquad\le 
 c_2 \exp\bigg(- c_3 \frac{\log (t)}{(mc_1)^{n/(s^+(n-1))}} \bigg) 
 = c_2 t^{-\frac{c_3}{(mc_1)^{n/(s^+(n-1))}}}.
\end{split}
\] 
Whenever we  choose $m>0$ to be so small that $\frac{c_3}{(mc_1)^{n/(s^+(n-1))}}>1$, we obtain
\[
 \begin{split}
 &\int_1^\infty \Ha^{n-\alpha(\cdot)}_\infty \Bigg(\bigg\{x \in D : \tilde I_{\s} \bigg(\frac{ |\nabla u|}{2(1+|D|}\bigg) > \bigg(\frac{\log (t)}{(mc_1)^{n/(s^+(n-1))}} \bigg)^{\frac{s^+(n-1)}{n}}\bigg\}\Bigg)\, dt\\
 &\qquad \le \int_1^\infty c_2 t^{-\frac{c_3}{(mc_1)^{n/(s^+(n-1))}}} \, dt =: c_4< \infty.
 \end{split}
 \]
Now the claim follows by choosing 
\[
a=\biggl(\frac{m}{2(1+\vert D\vert )}\biggr)^{\frac{n}{s^{+}(n-1)}} \mbox {and }
b= \Ha^{n-\alpha(\cdot)}_\infty(D) + c_4\,. \qedhere
\]
 \end{proof}


\bibliographystyle{amsalpha}

\end{document}